\documentclass{amsart}

\usepackage[a4paper,hmargin=3.5cm,vmargin=4cm]{geometry}
\usepackage{amsfonts,amssymb,amscd,amstext}
\usepackage{graphicx}
\usepackage[dvips]{epsfig}
\usepackage{mathpazo}
\usepackage{enumerate}

\def\Kcal{\mathcal{K}}
\def\Ncal{\mathcal{N}}
\def\Mcal{\mathcal{M}}
\def\Hcal{\mathcal{H}}
\def\C{\mathbb{C}}\def\c{\mathbb{C}}
\def\d{\mathbb{D}}
\def\R{\mathbb{R}}\def\r{\mathbb{R}}
\def\Z{\mathbb{Z}}\def\z{\mathbb{Z}}
\def\N{\mathbb{N}}\def\n{\mathbb{N}}
\newtheorem{theorem}{Theorem}[section]

\newtheorem{claim}[theorem]{Claim}
\newtheorem{lemma}[theorem]{Lemma}
\newtheorem{corollary}[theorem]{Corollary}
\newtheorem{remark}[theorem]{Remark}

\theoremstyle{definition}
\newtheorem{definition}[theorem]{Definition}

\numberwithin{equation}{section}

\begin{document}

\title[Complete minimal surfaces and harmonic functions]
{Complete minimal surfaces and harmonic functions}

\author[A.~Alarc\'{o}n]{Antonio Alarc\'{o}n$^\dagger$}
\address{Departamento de Matem\'{a}tica Aplicada \\ Universidad de Murcia
\\ E-30100 Espinardo, Murcia \\ Spain}
\email{ant.alarcon@um.es}

\author[I.~Fern\'{a}ndez]{Isabel Fern\'{a}ndez$^\ddagger$}
\address{Departamento de Matem\'{a}tica Aplicada I \\ Universidad de Sevilla
\\ E-41012 Sevilla \\ Spain}
\email{isafer@us.es}

\author[F.J.~L\'{o}pez]{Francisco J. L\'{o}pez$^\dagger$}
\address{Departamento de Geometr\'{\i}a y Topolog\'{\i}a \\
Universidad de Granada \\ E-18071 Granada \\ Spain}
\email{fjlopez@ugr.es}

%\date{\today}

\thanks{$^\dagger$ Research partially
supported by MCYT-FEDER research project MTM2007-61775 and Junta
de Andaluc\'{i}a Grant P06-FQM-01642}
\thanks{$^\ddagger$ Research partially
supported by MCYT-FEDER research project MTM2007-64504 and Junta
de Andaluc\'{i}a Grant P06-FQM-01642}


\begin{abstract} We prove that for any open Riemann surface $\mathcal{N}$ and any non constant harmonic function $h:\mathcal{N} \to \r,$
there exists a complete conformal minimal immersion $X:\mathcal{N} \to \r^3$ whose third coordinate function coincides with  $h.$

As a consequence, complete minimal surfaces with arbitrary conformal structure and whose Gauss map misses two points are constructed.
\end{abstract}

\maketitle

\thispagestyle{empty}

\section{Introduction}\label{sec:intro}
Conformal minimal immersions of Riemann surfaces in $\r^3$ are harmonic  maps. This basic fact has strongly influenced the global theory of minimal surfaces, supplying this field with powerful tools coming from classical complex analysis and Riemann surfaces theory.

If $X=(X_1,X_2,X_3):\mathcal{N} \to \r^3$ is conformal and minimal,  the holomorphic 1-forms $\phi_j:=\partial X_j,$ $j=1,2,3,$ satisfy the equation  $\phi_1^2+\phi_2^2+\phi_3^2=0.$ As a consequence, any conformal minimal immersion is uniquely determined (up to translations) by any two of its harmonic coordinate functions. On the other hand, it is reasonable to think that the family of conformal minimal immersions {\em with a prescribed coordinate function} is in general vast. However, the construction of this kind of surfaces turns out to be more complicated than expected under completeness assumptions. A pioneering result in this direction can be found in \cite{af}, where a satisfactory answer in the simply connected case is given. The aim of this paper is to extend this result to the more general setting of arbitrary open Riemann surfaces.

Our main theorem asserts that:

\begin{quote}
{\bf Theorem I.} {\em Let $\mathcal{N}$ be an open Riemann surface, let $h:\Ncal\to\R$ be a non constant harmonic function and let $p:\Hcal_1(\Ncal,\Z)\to\R^3$ be a group morphism such that the third coordinate of $p(\gamma)$ coincides with $\mbox{Im}\int_{\gamma} \partial h$, for all $\gamma\in\Hcal_1(\Ncal,\Z)$.

Then there exists a complete conformal minimal immersion $X=(X_1,X_2,X_3):\Ncal\to\R^3$ with $X_3=h$ and flux map $p_X=p.$}
\end{quote}
Recall that the flux map of a conformal minimal immersion $X:\mathcal{N} \to \r^3$ is given by $p_X(\gamma)=\mbox{Im}\int_{\gamma} \partial X$, for all $\gamma\in\Hcal_1(\Ncal,\Z)$.

It may be followed from Theorem I some interesting results concerning the Gauss map of minimal surfaces, the Calabi-Yau problem,  holomorphic null curves in $\c^3$ and maximal surfaces in the Lorentz-Minkowski space $\r^3_1.$

The study of the Gauss map is one of the fundamental problems in the  theory of minimal surfaces. Fujimoto \cite{fujimoto} showed that the  number of exceptional values of the Gaussian image of a complete non flat minimal surface is at most four, improving some classical results by Osserman \cite{oss} and Xavier \cite{xavier}. Since Sherk's minimal surfaces omit four points, then  Fujimoto's theorem is sharp. However, the number of exceptional values strongly depends on the underlying conformal structure. For instance, by Picard's Theorem there are no conformal non flat minimal immersions of the complex plane in $\r^3$ whose Gauss map omits three points. So it is natural to wonder whether any open Riemann surface admits a complete conformal minimal immersion with Gauss map omitting two points. We answer affirmatively this question, proving considerably more:

\begin{quote}
{\bf Theorem II.} {\em Let $\mathcal{N}$ be an open Riemann surface, and let $p:\Hcal_1(\Ncal,\Z)\to\R^3$ be a group morphism.

Then there exists a complete conformal minimal immersion $X:\Ncal\to\R^3$ whose Gauss map omits two antipodal points and $p_X=p.$}
\end{quote}

Calabi-Yau conjectures deal with the existence problem of complete minimal surfaces with bounded coordinate functions. There is large literature on this topic, see \cite{jorge-xavier,nadi,c-m,f-m-m} for a good setting. From Theorem I follows that a (necessary and) sufficient condition for an open Riemann surface to admit a complete conformal non flat minimal immersion into an open slab of $\r^3$ is to carry non constant bounded harmonic functions (see Corollary \ref{co:slab1}).

Likewise, by Theorem I, any open Riemann surface $\mathcal{N}$ carrying a non constant holomorphic function $f:\mathcal{N}\to\c$ admits a complete null holomorphic immersion $(F_1,F_2,F_3):\mathcal{N}\to\c^3$ (and so a complete holomorphic immersion $(F_1,F_3):\mathcal{N}\to\c^2$) with $F_3=f.$ The family of open Riemann surfaces admitting non constant bounded holomorphic functions is particularly interesting from several points of view. This space contains examples of arbitrary open topological type, and as above any such surface admits a complete null holomorphic immersion in $\c^2\times \d$ (and so a complete holomorphic immersion in $\c\times\d$). We have compiled these ideas in the following result (for the construction of proper complete null curves in $\c^2\times\d$ and proper complete holomorphic curves in $\c\times\d$ see Corollary \ref{co:slab2}):
\begin{quote}
{\bf Corollary III.} {\em Let $M$ be an open orientable surface. Then there exists a complete minimal surface homeomorphic to $M$  all whose associate surfaces are well defined and contained in a slab of $\r^3.$}
\end{quote}
Simply connected bounded  complete null curves immersed in $\c^3$ and bounded  complete holomorphic curves immersed in $\c^2$ with arbitrary finite genus has been constructed in \cite{muy1} and \cite{muy2}, respectively. Moreover, complete minimal surfaces properly immersed in an open slab of $\r^3$ of arbitrary topological type can be found in \cite{f-m-m} (see also \cite{jorge-xavier,r-t,lop-ori,lop-slab,m-m,a-f-m} for a good setting). The problem of constructing bounded complete null holomorphic curves in $\c^3$ with  nontrivial topology remains open.

Finally, Theorem I provides weakly complete conformal maximal immersions in the Lorentz-Minkowski 3-spacetime $\r^3_1$ with
singularities and prescribed  spacelike or timelike coordinate functions (the notion of weakly complete maximal surface with singularities was defined in \cite{u-y}). See Corollary \ref{co:maxi} for more details.

In a forthcoming paper \cite{al2}, the authors will extend these results to the nonorientable setting.
%%%%%%%%%%%%%%%
\section{Preliminaries}

For a topological surface $M,$ we will denote as  $\partial(M)$ the one dimensional topological manifold determined by the boundary points of $M.$ Given $S \subset M,$ $S^\circ$ and $\overline{S}$ will denote the interior and the closure of $S$ in $M,$ respectively.
A Riemann surface $M$ is said to be {\em open} if it is non-compact and $\partial(M)=\emptyset.$
\begin{remark}
In the sequel $\Ncal$ will denote a fixed but arbitrary open Riemann surface, $W\subset\Ncal$ an open connected subset of finite topology, and $S\subset W$ a compact set. 
\end{remark}

For a proper subset $M$ of $\Ncal$ we will denote by $\Omega_0(M)$ as the space of holomorphic
1-forms on an open neighborhood of $S$ in $\Ncal,$ whereas $\Omega_0^*(M)$ will denote
the space of complex 1-forms $\theta$ of type $(1,0)$ that are continuous on $M$ and holomorphic on $M^\circ$. As usual, a 1-form $\theta$ on $M$ is said to be of type $(1,0)$ if for any conformal chart $(U,z)$ in $ \mathcal{ N},$ $\theta|_{U \cap M}=h(z) dz$ for some function $h:U \cap M \to \c.$

\begin{definition}[Admissible set]
A compact subset $S\subset W$ is said to be admissible in $W$ if and only if:
\begin{itemize}
\item $W-S$ has no bounded components in $W$ (by definition, a connected component $V$ of $W-S$ is said to be {\em bounded} in $W$ if $\overline{V}\cap W$ is compact, where $\overline{V}$ is the closure of $V$ in $\Ncal$),
\item $M_S:=\overline{S^\circ}$  consists of a finite collection of pairwise disjoint compact regions in $W$ with   $\mathcal{ C}^0$ boundary,
\item $C_S:=\overline{S-M_S}$ consists of a finite collection of pairwise disjoint analytical Jordan arcs (recall that a compact Jordan arc in $\Ncal$ is said to be analytical if it is contained in an open analytical Jordan arc in $\Ncal$), and
\item any component $\alpha$ of $C_S$  with an endpoint  $P\in M_S$ admits an analytical extension $\beta$ in $W$ such that the unique component of $\beta-\alpha$ with endpoint $P$ lies in $M_S.$
\end{itemize}
\end{definition}
Observe that if $S$ is admissible in $\Ncal$ then it is admissible in $W$ as well, but the contrary is in general false.\\

%%%%%%

With the previous notation, a function $f:S\to\C$ defined on an admissible set $S$ in $W$ is said to be {\em smooth} if $f|_{M_S}$
admits a smooth extension $f_0$ to an open domain $V\subset W$ containing $M_S,$ and for any component $\alpha$ of $C_S$
and any open analytical Jordan arc $\beta$ in $W$ containing $\alpha,$  $f$ admits an smooth extension $f_\beta$ to $\beta$
satisfying that $f_\beta|_{V \cap \beta}=f_0|_{V \cap \beta}.$

Likewise, a 1-form $\theta\in \Omega_0^*(S)$ is said to be {\em smooth} if,
for any closed conformal disk $(U,z)$ on $W$ such that $ S\cap U$ is admissible in $W,$ $\theta/dz$ is smooth in the previous sense.

Given a smooth function $f:S\to\C$ holomorphic on $S^\circ$, we set $df \in \Omega_0^*(S)$
as the smooth 1-form given by $df|_{M_S}=d (f|_{M_S})$ and $df|_{\alpha \cap U}=(f \circ \alpha)'(x)dz|_{\alpha \cap U},$
where $(U,z=x+i y)$ is a conformal chart on $W$ such that $\alpha \cap U=z^{-1}(\R \cap z(U)).$
Obviously, $df|_\alpha(t)= (f\circ\alpha)'(t) dt$ for any component $\alpha$ of $C_S,$ where $t$ is any smooth parameter along $\alpha.$
A smooth 1-form $\theta \in \Omega_0^*(S)$ is said to be {\em exact} if $\theta=df$ for some smooth $f:S\to\C$ holomorphic on $S^\circ$,
or equivalently if $\int_\gamma \theta=0$ for all $\gamma \in \mathcal{ H}_1(S,\Z).$\\

%%%%%%%%%%

The following Lemma and its Corollaries will be required to approximate minimal immersions by immersions defined on larger domains (possibly with higher topology).

\begin{lemma}[\text{\cite[Approximation Lemma]{al}}]\label{lem:runge}
Let $S$ be an admissible compact set in $W,$ and $\Phi=(\phi_j)_{j=1,2,3}$ a smooth triple in $\Omega_0^\ast(S)^3$,
 such that $\sum_{j=1}^3\phi_j^2=0$, $\sum_{j=1}^3|\phi_j|^2$ never vanishes on $S$, and $\Phi|_{M_S}\in\Omega_0(M_S)^3$.

Then $\Phi$ can be uniformly approximated on $S$ by a sequence $\{\Phi_n=(\phi_{j,n})_{j=1,2,3}\}_{n\in\N}$ in $\Omega_0(W)^3$ satisfying
\begin{enumerate}[$(i)$]
\item $\sum_{j=1}^3\phi_{j,n}^2=0,$
\item $\sum_{j=1}^3|\phi_{j,n}|^2$ never vanishes on $W$ and
\item $\Phi_n-\Phi$ is exact on $S$, for all $n\in\N.$
\end{enumerate}
\end{lemma}

Recall that a 1-form  $\theta \in \Omega_0^*(S)$ is said to be uniformly approximated on $S$ by 1-forms in $\Omega_0(W)$, if there exists  $\{\theta_n\}_{n \in \N}\subset\Omega_0(W)$ such that $\{\frac{\theta_n-\theta}{dz}\}_{n \in \N} \to 0$ uniformly on $S \cap U,$ for any conformal closed disc $(U,dz)$ on $W.$

\begin{corollary}[\text{\cite[Corollary 2.10]{al}}]\label{co:f3}
The sequence $\{\Phi_n=(\phi_{j,n})_{j=1,2,3}\}_{n\in\N}$ in the above lemma can be obtained so that
$\phi_{3,n}=\phi_3$ for all $n\in\N$, provided that $\phi_3$ extends holomorphically to $W$ and never vanishes on $C_S.$
\end{corollary}

\begin{corollary}\label{co:divisor}
The sequence $\{\Phi_n=(\phi_{j,n})_{j=1,2,3}\}_{n\in\N}$ obtained in Lemma \ref{lem:runge} can be chosen so that
 $\phi_{3,n}$ never vanishes on $W$, for all $n\in\N$, provided that $\phi_3$ never vanishes on $S$.
\end{corollary}

\begin{remark}
Although Corollary \ref{co:divisor} is not explicitly stated in \cite{al}, it can be deduced from the proof of the Approximation Lemma in \cite{al}.
Indeed, the $1$-form $\phi_{3,n}$ is defined as $\phi_{3,n}= e^{f_n} \psi_n$, where $f_n$ is a holomorphic function on $W$, and
$\psi_n\in\Omega_0(W)$ never vanishes on $W$ provided that $\phi_3$ does in $S,$ $n\in\n.$
\end{remark}

%%%%%%%%%%%%%%%%%%%%%%%%%%%%%%%%%%%%%%%%%%%%%%%%%%%%%%%%%%%%%%%%

\subsection{Minimal Surfaces} \label{sec:wei}

As remarked in Section \ref{sec:intro}, the coordinates functions of  a conformal minimal immersion $X=(X_1,X_2,X_3):W \to \R^3$ are harmonic.
If we denote  $\partial$ as the global complex operator given by
$\partial|_U=\frac{\partial}{\partial z} dz$ for any conformal chart $(U,z)$ on $W,$ then the corresponding $1$-forms $\phi_j=\partial X_j,$ $j=1,2,3,$ are holomorphic on $W$. Moreover, $X$ and its pull-back metric are given by
\begin{equation}\label{eq:X}
 X=\mbox{Re} \int (\phi_1,\phi_2,\phi_3),
\end{equation}
and
\begin{equation}\label{eq:metric}
ds^2_X=\sum_{k=1}^3 |\phi_k|^2
\end{equation} respectively. As a consequence, the triple  $\Phi=(\phi_1,\phi_2,\phi_3)$ satisfies the following properties:
\begin{enumerate}[$(i)$]
\item $\phi_k$ have no real periods, $k=1,2,3$,
\item $\sum_{k=1}^3 \phi_k^2=0,$
\item $\phi_k,$ $k=1,2,3,$ have no common zeroes.
\end{enumerate}
Conversely, given a vectorial holomorphic 1-form $\Phi=(\phi_1,\phi_2,\phi_3)$ on $W$ satisfying $(i)$ to $(iii)$, then \eqref{eq:X} determines a conformal minimal immersion $X:W \to \r^3$.

The triple $\Phi$ is said to be the Weierstrass representation of $X.$
A remarkable fact is that the stereographic projection of the Gauss map of $X$ is the (meromorphic) function $g=\frac{\phi_3}{\phi_1-i \phi_2}$. In particular, the poles and zeros of $g$ coincide with the zeros of $\phi_3$ with the same order (see \cite{osserman}).

The flux of $X$ along a closed curve $\gamma$ in $W$ is defined as $p_X(\gamma)=\int_\gamma \mu(s) ds,$ where $s$ is the arclength parameter of $\gamma$ and $\mu(s)$ is the conormal vector of $X$ at $\gamma(s)$ (i.e., the unique vector such that  $\{d X(\gamma'(s)),\mu(s)\}$ is an orthonormal positive basis of the tangent plane of $X$ at $\gamma(s)$). It is easy to check that $p_X(\gamma)=\mbox{Im}\int_{\gamma} \partial X$ and that the flux map $p_X:\mathcal{ H}_1(M,\z)\to \r^3$ is a group morphism.\\

%%%%%%%%%%%%%%%%%%%%%%%%%%%%%%%%%%%

As we will deal with admissible sets, a suitable notion for {\em minimal immersions} on admissible sets will be required.
This is the aim of the following definitions.

Let $S$ be a admissible subset in $W$ and $X:S\to\R^3$ a smooth map such that
$X|_{C_S}$ is regular, (i.e., $X|_\alpha$ is a regular curve for all $\alpha \subset C_S$).
By a {\em smooth normal field} along $C_S$ respect to $X$ we mean a field $\varpi:C_S \to \R^3$  such that,
for any analytical arc $\alpha \subset C_S$,
$\varpi\circ\alpha $ is smooth, unitary and orthogonal to $(X\circ\alpha)',$
$\varpi$ extends smoothly to any open analytical arc $\beta$ in $W$ containing $\alpha$,
and $\varpi$ is tangent to $X$ on $\beta \cap S.$
The normal field $\varpi$ is said to be {\em orientable} respect to $X$ if for any component $\alpha \subset C_S$ with endpoints $P_1,$ $P_2 \in \partial(M_S),$   and for any arclength parameter $s$ along $X|_\alpha,$   the basis $B_i=\{(X|_\alpha)'(s_i), \varpi(s_i)\}$ of the tangent plane of $X|_{M_S}$ at $P_i,$  $i=1,2,$ are both positive or negative, where $s_i$ is the value of $s$ for which  $\alpha(s_i)=P_i,$ $i=1,2.$

\begin{definition}
Given a proper subset $M\subset \Ncal,$ we denote by $\mathcal{ M}(M)$ the space of maps $X:M \to\r^3$ extending as a conformal minimal immersion to an open neighborhood of $M$ in $\Ncal.$
On the other hand, for an admissible set $S$ in $W$ we call  $\mathcal{ M}^*(S)$ as the  space of marked immersions  $X_\varpi:=(X,\varpi),$ where
\begin{enumerate}
\item $X:S\to\R^3$ is a smooth map,
\item $X|_{M_S} \in \mathcal{ M}(M_S)$,
\item $X|_{C_S}$ is regular, and
\item $\varpi$ is an  orientable smooth normal field along $C_S$ respect to $X$.
\end{enumerate}

We will endow $\mathcal{ M}(M)$ (resp. $\mathcal{ M}^\ast(S)$) with the $\mathcal{ C}^0$ topology of the uniform convergence on compact subsets of $M$ (resp. uniform convergence of maps and normal fields on $S$).
\end{definition}

The notions of Weierstrass data and flux map can be also extended to immersions in $\mathcal{ M}^\ast(S)$. Indeed, given $X_\varpi \in \mathcal{ M}^*(S),$ let $\partial X_\varpi=(\hat{\phi}_j)_{j=1,2,3}$ be the complex vectorial {\em 1-form}  on $S$ given by
 $\partial X_\varpi:=\partial (X|_{M_S}),$
and for any component $\alpha$ of $C_S,$  $\partial X_\varpi:= dX(\alpha'(s)) + i \varpi(s),$
where $s$ is the arclength parameter of $X|_\alpha$ such that $\{dX (\alpha'(s_0)), \varpi(s_0)\}$ is positive provided that  $\alpha(s_0) \in \partial(M_S).$

The triple $\hat\Phi:=\partial X_\varpi$ will be called the {\em generalized Weierstrass data} of $X_\varpi.$
It is clear that $\hat{\Phi} \in \Omega_0^*(S)^3$ and is smooth.
Notice also that $\sum_{j=1}^3 \hat{\phi}_j^2=0,$ $\sum_{j=1}^3 |\hat{\phi}_j|^2$ never vanishes on $S$ and
$\mbox{Real} (\hat{\phi}_j)$ is an {\em exact} real 1-form on $S,$
$j=1,2,3,$ hence we also have $X(P)=X(Q)+\mbox{Real} \int_{Q}^P
(\hat{\phi}_j)_{j=1,2,3},$ $P,$ $Q \in S.$
In particular, since $X|_{M_S} \in \mathcal{ M} (M_S)$ then $(\phi_j)_{j=1,2,3}:=(\hat{\phi}_j|_{M_S})_{j=1,2,3}$
are the Weierstrass data of $X|_{M_S}$.

The group homomorphism
 $$p_{X_\varpi}:\mathcal{ H}_1(S,\z) \to \r^3, \quad p_{X_\varpi}(\gamma)=\mbox{Im} \int_\gamma \partial X_\varpi,$$
is said to be the {\em generalized flux map} of $X_\varpi.$ Obviously,  $p_{X_{\varpi_Y}}=p_Y|_{\mathcal{ H}_1(S,\z)}$ provided that $X=Y|_{S}$ and $\varpi_Y$ is the conormal field of $Y\in \mathcal{M}(W)$ along any curve in $C_S.$

%%%%%%%%%%%%%%%%%%%%%%%%%%%%%%%%%%%%%%%%%%%%%%%%%%%%%%%%%%%%%%%%%%%%%%%%%%%%%%%%%%

\section{The Completeness Lemma}

Given a compact subset $M \subset \mathcal{ N}$ and a map $X=(X_1,X_2,X_3):M \to \r^3,$ we denote $\|X\|:= \max_{M} \big\{ \big(\sum_{j=1}^3 X_j^2\big)^{1/2}\big\}$ as the maximum norm of $X$ on $M.$

The following lemma concentrates most of the technical computations required in the proof of the main result of this paper.

\begin{lemma}\label{lem:lema}

Let $U,V$ be two compact regions in $\Ncal$ such that  $U\subset V^\circ$ and $V^\circ-U$ has no bounded components in $V^\circ.$ Consider a non constant harmonic function $h:V\to\r,$ an immersion $X=(X_1,X_2,X_3)\in\mathcal{M}(U)$ and a group morphism $p:\Hcal_1(V,\Z)\to\R^3$ such that $X_3=h|_{U},$ $p_X=p|_{\mathcal{H}_1(U,\z)}$ and the third coordinate of $p(\gamma)$ is $\mbox{Im}\int_\gamma \partial h,$ $\forall \gamma\in\mathcal{H}_1(V,\z).$

Then, for any $P_0\in U$ and $\epsilon>0,$ there exists $Y=(Y_1,Y_2,Y_3)\in\mathcal{M}(V)$ satisfying that:
\begin{enumerate}[$(i)$]
\item $\|Y-X\|<\epsilon$ on $U,$
\item $Y_3=h,$
\item $p_Y=p$ and
\item ${\rm dist}_Y(P_0,\partial (V)) > 1/\epsilon.$
\end{enumerate}
Here ${\rm dist}_Y$ denotes the distance on $V$ in the intrinsic metric of the immersion $Y.$
\end{lemma}

\begin{proof}

We will prove this lemma by induction on (minus) the Euler characteristic of $V^\circ-U$ (recall that, since we are assuming that $V^\circ-U$ has no bounded components in $V^\circ,$ then $\chi(V^\circ-U)\leq 0$). The induction process is enclosed in the following two claims.

%%%%%%%%%%%%%%
\begin{claim}
The lemma holds if $\chi(V^\circ-U)=0.$
\end{claim}
\begin{proof}
The argument we use now is analogous to the one employed in \cite[Lemma 1]{af}. Write $V^\circ-U=\cup_{j=1}^k A_j$, where $A_j$ are pairwise disjoint open annuli.
On each component $A_j$ we define the following labyrinth of compact sets. Let $z_j:A_j\to\C$ be a conformal parametrization,
and consider a compact region $C_j\subset A_j$ such that $C_j$ contains no zeros of $\partial h,$ $z_j(C_j)$ is a compact annulus of radii $r_j$ and $R_j,$ where $r_j<R_j,$ and $z_j(C_j)$ contains the homology of $z_j(A_j).$ Write $\phi_3=\partial X_3=f_j(z_j)dz_j$, with $|f_j|>0$ on $C_j.$ Let $\mu$ be a positive constant with $$\mu<\min\{|f_j(P)|\;|\; P\in C_j,\; j=1,\ldots,k\}.$$
Fix a natural $N$ (which will be specified later) such that
$2/N<\min\{R_j-r_j\;|\;j=1,\ldots,k\}.$ For any $n\in\{1,\ldots,2N^2\}$, consider the compact set in $C_j:$
\[
\Kcal_{j,n}=\left\{ p\in A_j\; \left|\; s_n+\frac1{4N^3}\leq |z_j(p)|\leq
s_{n-1}-\frac1{4N^3},\right.\right.
\left.\frac1{N^2}\leq {\rm arg}((-1)^{n}z_j(p))
\leq 2\pi-\frac1{N^2}\right\},
\]
where $s_n:=R_j-n/N^3$. Then, define
$$
\Kcal_j=\bigcup_{n=1}^{2N^2}\Kcal_{j,n}\qquad \mbox{and}\qquad  \Kcal=\bigcup_{j=1}^k \Kcal_j.
$$

Define $\Phi\in \Omega_0(\overline{U}\cup \Kcal)^3$ by
$${\Phi}|_{ U}=\partial X, \qquad \Phi|_{\Kcal}= \Big(\frac{1}{2}(\frac{1}{M}  - M )\,\phi_3,\, \frac{i}{2}(\frac{1}{M} + M)\,\phi_3,\, \phi_3\Big), $$
where $M>2N^4$ is a constant.

By Corollary \ref{co:f3} applied to $S=U\cup \mathcal{K},$ $\Phi,$ and an open tubular neighborhood of $V,$ we can infer the existence of $\Psi\in\Omega_0(V)^3$ giving rise to a well-defined conformal minimal immersion
$Y=(Y_1,Y_2,Y_3)\in\mathcal{ M}(V)$ fulfilling $(i),$ $(ii)$ and $(iii),$ and whose metric $ds^2_Y$ satisfies
\begin{equation}\label{eq:aprox}
ds_Y^2 > \frac{1}{4}\big(\frac{1}{M} + M\big)^2\mu^2 |dz_j|^2 > N^8\mu^2 |dz_j|^2 \qquad \mbox{on}\quad \Kcal_j, \; j=1,\ldots,k.
\end{equation}

To finish the claim it remains to check $(iv)$. Taking into account that $ds_Y^2\geq |\phi_3|^2 >\mu^2|dz_j|^2 $ on $C_j$, and \eqref{eq:aprox}, it is not hard to check that there exists a positive constant $\rho_j$ depending neither on $\mu$ nor $N$ such that
$$
\mbox{length}_{ds_Y^2}(\alpha)>\rho_j\cdot \mu\cdot N
$$
for any $\alpha$ curve in $C_j$ joining the two components of $\partial(C_j)$. Thus, we can choose $N$ large enough so that
$\rho_j\cdot \mu\cdot N>1/\epsilon$ for any $j=1,\ldots,k$. In particular, $(iv)$ is achieved.
\end{proof}

%%%%%%%%%%%%%%%

\begin{claim} Let $n>0$. Assume that the Lemma holds if $-\chi(V-U^\circ)<n$. Then it also holds for $-\chi(V-U^\circ)=n$.
\end{claim}
\begin{proof}
 Since $-\chi(V-U^\circ)>0$, there exists $\hat\gamma\in\Hcal_1(V,\Z)-\Hcal_1(U,\Z)$ intersecting $V-U^\circ$
in a Jordan arc $\gamma$ with endpoints $P_1,P_2\in\partial (U)$ and otherwise disjoint from $\partial (U),$ and such that $S:=U\,\cup\,\gamma$ is an admissible set in an open tubular neighborhood $W$ of $V$ in $\Ncal.$ Moreover, we take $\hat{\gamma}$ so that $\partial h$ never vanishes on $\gamma.$

Take $F_\varpi\in\Mcal^\ast(S)$, $F=(F_1,F_2,F_3)$, satisfying
$F|_U=X$, $F_3=h|_S$, the third coordinate of $\partial F_\varpi$ is $\partial h|_S,$  and $p_{F_\varpi}(\hat\gamma)=p(\hat\gamma)$.

By Corollary \ref{co:f3} applied to the (generalized) Weierstrass data of $F_{\varpi},$ $S$ and $W,$ we obtain a compact tubular neighborhood  $W'$ of $S$ in $V^\circ$ and $Z=(Z_1,Z_2,Z_3)\in\mathcal{M}(W')$ such that
$\|Z-X\|<\epsilon/2$ on $U$, $p_{Z}=p|_{\mathcal{H}_1(W',\z)}$, and $Z_3=h|_{W'}$. Since $-\chi(V-(W')^\circ)<n$, the induction hypothesis applied to $Z$ and $\epsilon/2$ gives the existence of an immersion $Y$ satisfying the conclusion of the Lemma.
\end{proof}

The proof is done.
\end{proof}

%%%%%%%%%%%%%%%%%%%%%%%%%%%%%%%%%%%%%%%%%%%%%%%%%%%%%%%%%%%%%%%%%%%%%%%%%%%%%%%%%%%%%%%%%%%%%%%%%%%%%%%%%%%%%%%%%%%%%%%%%%%%%%%%%%%%%%%%%%%%

\section{Main Results}
In this section we prove the results stated in the introduction and obtain some corollaries.
\begin{theorem}\label{th:harmonic}
Let $h:\Ncal\to\R$ be a non constant harmonic function and $p:\Hcal_1(\Ncal,\Z)\to\R^3$ a group morphism such that the third coordinate of $p(\gamma)$ coincides with $\mbox{Im}\int_{\gamma} \partial h$, for all $\gamma\in\Hcal_1(\Ncal,\Z)$.\\
Then there exists a complete conformal minimal immersion $X=(X_1,X_2,X_3):\Ncal\to\R^3$ with $X_3=h$ and $p_X=p$.
\end{theorem}

\begin{proof}
Consider an exhaustive sequence $\{V_n\}_{n\in\N}\subset\Ncal$ of compact regions such that $V_1$ is simply connected,
 ${V_{n-1}}\subset V_{n}^\circ $, and $V_{n}^\circ  -{V_{n-1}}$ has no bounded components in $V_{n}^\circ$, $n\geq 2 $.

Let $Y_1\in\mathcal{ M}(V_1)$ be the conformal minimal immersion with Weierstrass data given by
$\phi_3=(\partial h)|_{V_1}$ and $g=\phi_3/dz$, where $z$ is a conformal parameter on $V_1$.

Fix a point $P_0\in V_1^\circ,$ % and a sequence $\{\sigma_n\}_{n\in\N}$ with $0<\sigma_n<1,$ $\prod_{i\geq 1} \sigma_n=1/2$,
 and apply recursively Lemma \ref{lem:lema} to obtain a sequence $\{Y_n\}_{n\in\N}$, $Y_n\in\mathcal{ M}(V_n)$ satisfying that:
\begin{enumerate}[$a)$]
\item $||Y_n-Y_{n-1}||<1/n^2$ on $V_{n-1},$
%\item $ds^2_{X_n} \geq \sigma_n\, ds^2_{X_{n-1}}$ in $V_{n-1}$, ($ds^2_{X_n}$ denotes the metric of $X_n$),
\item ${\rm{dist}}_{Y_n}(p_0,\partial(V_n))>n^2,$
\item $p_{Y_n}=p|_{\mathcal{H}_1(V_n,\z)}$, and
\item the third coordinate function of $Y_n$ coincides with $h|_{V_n}$,
\end{enumerate}
 for all $n\in\N$. Here ${\rm{dist}}_{Y_n}$ denotes the distance on $V_n$ in the intrinsic metric of the immersion $Y_n.$
Since $\Ncal=\cup_{n\in\N} V_n$, property $a)$ gives that  $\{Y_n\}_{n\in\N}$ converges to a harmonic limit map
$X=(X_1,X_2,X_3):\Ncal\to\R^3$ uniformly on compact sets (Harnack Theorem).
Moreover, from Hurwitz Theorem and the fact that $\partial Y_n$ never vanishes we infer that either $X$ degenerates on a point or has no branch points. From $d)$ follows $X_3=h$ which is non constant and so the first possibility can not occur. On the other hand, properties $b)$ and $c)$ give that $Y$ is complete and $p_X=p,$ respectively.
\end{proof}

%%%%%%%%%%%%%%%%%%%

Any open Riemann surface carries regular harmonic functions, that is to say, harmonic functions with never vanishing differential. As a consequence, any open Riemann surface admits a conformal complete minimal immersion in $\r^3$ whose Gauss map misses to antipodal values. For completeness we include a detailed proof of all these facts based  in Corollary \ref{co:divisor}.

\begin{theorem}\label{th:sinceros}
Let $p:\Hcal_1(\Ncal,\Z)\to\R^3$ be a group morphism.

Then there exists a complete conformal minimal immersion $X:\Ncal\to\R^3$ such that its meromorphic Gauss map has neither zeros nor poles and $p_X=p.$
\end{theorem}

\begin{proof}
%We will show first that it is possible to construct a holomoprhic $1$-form $\theta$ on $\Ncal$ without zeroes and with purely imaginary periods.
%Although this is a classical known result, we will give here a proof based on Corollary \ref{co:divisor} for the sake of completeness.
%Moreover, we can choose $\theta$ so that $-i\int_\gamma\theta$ agrees with the third coordinate of $p(\gamma)$, for all $\gamma\in\Hcal_1(\Ncal,\Z)$.

Take $\{V_n\}_{n\in\N}\subset \Ncal$ an exhaustive sequence of compact regions
such that $V_1$ is simply connected, ${V_n}\subset V_{n+1}^\circ,$ $V_{n+1}^\circ-V_n$ has no bounded components and $\chi(V_{n+1}^\circ-V_n)=-1.$ Let $F\in\mathcal{ M}(V_1)$ be a conformal minimal immersion with Weierstrass data
$\Psi=(\psi_{1},\psi_{2},\psi_{3})$ such that $\psi_{3}$ never vanishes on $V_1.$

%Take a sequence $\{\sigma_n\}_{n\in\N}$ with $0<\sigma_n<1$, $\prod_{i\geq 1}\sigma_n=1/2$.
Fix $\epsilon>0.$ The key step in the proof is the construction of a sequence $\{Y_n\}_{n\in\N}$, $Y_n\in\mathcal{ M}(V_n)$ with Weierstrass data $\Phi_n=\{(\phi_{j,n})_{j=1,2,3}\}$ satisfying that:
\begin{enumerate}[$a)$]
\item $\|Y_{n} - Y_{n-1}\|<\epsilon/n^2$ on $V_{n-1}$,
%\item $ds^2_{X_n}\geq \sigma_n\, ds^2_{X_{n-1}}$ in $V_{n-1}$,
\item $p_{Y_n}=p|_{\mathcal{H}_1(V_n,\z)}$ and
\item $\phi_{3,n}$ never vanishes on $V_n,$% and
%\item $|\phi_{3,n}/dz| \geq \mu$ in $V_1$,
\end{enumerate}
for all $n\geq 2$.

Indeed, choose $Y_1=F$ and assume that we have constructed $Y_1,\dots,Y_n.$
Then the immersion $Y_{n+1}$ is defined as follows.
Let $\hat\gamma\in\Hcal_1(V_{n+1},\Z)-\Hcal_1(V_n,\Z)$ intersecting $V_{n+1}-V_n^\circ$
in a Jordan arc $\gamma$ with endpoints $P_1,P_2\in\partial (V_n)$ and otherwise disjoint from $\partial (V_n),$ and such that $S:=V_n\,\cup\,\gamma$ is an admissible set in an open tubular neighborhood $W$ of $V_{n+1}$ in $\Ncal.$ Then extend $Y_n$ to a marked immersion $Z_\varpi\in\mathcal{ M}^\ast(S)$ satisfying that $p_{Z_\varpi}=p|_{\mathcal{H}_1(S,\z)}$
and the third coordinate of $\partial Z_\varpi$ never vanishes on $\gamma$. Applying Corollary \ref{co:divisor}
to the generalized Weierstrass data of $Z_\varpi,$ $S$ and $W,$
and integrating the resulting  $1$-forms we get $Y_{n+1}\in\mathcal{ M}(V_{n+1})$
satisfying the desired conditions.

By $a),$ Harnack theorem and Hurwitz theorem, the sequence $\{Y_n\}_{n\in\N}$ converges uniformly on compact sets to a conformal minimal immersion $Y:\Ncal\to\R^3,$ provided that $\epsilon$ is small enough. Label $\Phi=(\phi_1,\phi_2,\phi_3)$ as its Weierstrass data.
It is clear that $p=p_Y$, let us check now that $\phi_3$ never vanishes.
Indeed, assume $\phi_3$ has a zero at a point in $V_{n_0},$ for $n_0\in\n.$ Since $\phi_{3,n}$ never vanishes in $V_{n_0}$ for all $n\geq n_0,$ then $\phi_3$ vanishes identically on $V_{n_0}$ (Hurwitz theorem) and so in $\Ncal$. However, from $a)$ we infer that $\|Y-Y_1\|\leq\epsilon\sum_{n=1}^\infty 1/n^2=\epsilon\pi^2/6$ and so the third coordinate of $Y$ is non constant provided that $\epsilon$ is small enough, a contradiction.

Set $h:\Ncal\to\r$ by $h(P)=\mbox{Re}\int_{P_0}^P\phi_3$, where $P_0$ is an arbitrary fixed point in $\Ncal.$ Applying Theorem \ref{th:harmonic} to $h$ and $p$ we obtain a {\em complete} conformal minimal immersion $X=(X_1,X_2,X_3):\mathcal{N}\to\r^3$ such that $p_X=p$ and $X_3=h.$ As $\partial X_3=\phi_3$ never vanishes on $\Ncal$ then the meromorphic Gauss map of $X$ has neither zeros nor poles, concluding the proof.
%Thus, $\theta:=\phi_3$ is a holomorphic $1$-form, never vanishing on $\Ncal$, and with the desired imaginary periods.
%Set $u(P)=\mbox{Re}\int_{P_0}^P\theta$, where $P_0$ is an arbitrary fixed point in $\Ncal$.
%Theorem \ref{th:harmonic} finishes the proof.
\end{proof}

Open Riemann surfaces carrying non constant bounded harmonic functions are hyperbolic, but the reciprocal is false in general. However, in the case of finite topology  both statements are equivalent. Even more, if $\mathcal{N}$ is biholomorphic to a compact Riemann surface minus a finite collection of at least two pairwise disjoint closed discs, then there exists proper harmonic maps $h:\mathcal{N} \to (0,1).$ As a consequence,

\begin{corollary} \label{co:slab1}
Any of the following statements holds:
\begin{enumerate}
\item[(a)]  $\mathcal{N}$ carries a non constant bounded harmonic function if and only if there exists a conformal complete non flat minimal immersion of $\mathcal{N}$ in a horizontal slab of $\r^3.$
\item[(b)] If $\mathcal{N}$ is hyperbolic and of finite topology, then there exists a conformal complete non flat minimal immersion of $\mathcal{N}$ in a horizontal slab of $\r^3.$
\item[(c)] If $\mathcal{N}$ is biholomorphic to a compact Riemann surface minus a finite collection of at least two pairwise disjoint closed discs, then $\mathcal{N}$ admits a proper conformal complete non flat minimal immersion in an open horizontal slab of $\r^3.$
\end{enumerate}
In addition, in any case the first two coordinates of the flux map can be prescribed.
\end{corollary}

If $h$ is the real part of a non constant holomorphic function and $p=0,$ Theorem \ref{th:harmonic} also gives that:

\begin{corollary} \label{co:slab2}
Any of the following statements holds:
\begin{enumerate}
\item[(d)] The following assertions are equivalent:
\begin{itemize}
\item $\mathcal{N}$ carries a non constant bounded holomorphic function.
\item There exists a full\footnote[1]{A complex curve in $\c^n$ is said to be full if it is not contained in a linear complex subspace.} complete null immersion of $\mathcal{N}$ in $\c^2\times\d.$
\item There exists a full complete holomorphic immersion of $\mathcal{N}$ in $\c\times\d.$
\end{itemize}
\item[(e)] If $\mathcal{N}$ is hyperbolic and of finite topology, then there exists a full complete null immersion of $\mathcal{N}$ in $\c^2\times\d$ and a full complete holomorphic immersion of $\mathcal{N}$ in $\c\times\d.$
\item[(f)] If $\mathcal{N}$ admits a proper holomorphic function into the unit disk, then $\mathcal{N}$ admits a full proper complete minimal immersion in $\c^2\times D$ and a full proper complete holomorphic immersion in $\c\times D,$ where $D$ is any simply connected planar domain (the case $D=\c$ is proved in \cite{al}).
\end{enumerate}
\end{corollary}

\begin{remark}\label{rem:bounded}
The family of Riemann surfaces involved in item $(d)$ (and so in item $(a)$) contains examples with any open orientable topological type.

The family of Riemann surfaces concerning item $(f)$ is also very vast. For instance, it includes all the coverings of the unit disk.
\end{remark}

Although the first statement of the above remark is well known, for completeness we sketch a proof based on Scheinberg approximation results \cite{sc}. Let $\mathcal{N}$ be an open Riemann surface, and consider two compact regions $M,$ $V\subset \mathcal{N}$ such that $M\subset V^\circ,$ $\chi(V^\circ-M)=-1$ and $V^\circ-M$ has no bounded components in $V^\circ.$ Take also $\epsilon>0$ and a non constant holomorphic function  $f:M\to\d.$
Consider a Jordan arc $\gamma\subset V^\circ-M$ with endpoints in $\partial (M)$ and otherwise disjoint from $\partial (M)$ such that $\chi(V^\circ-(M\cup \gamma))=0$ and $V^\circ-(M\cup \gamma)$ has no bounded components in $V^\circ.$ For simplicity write $S=M\cup \gamma.$
Construct a continuous function $\hat{f}:S\to\d$ with $\hat{f}|_{M}=f,$ and use Scheinberg approximation Theorem to find a compact tubular neighborhood $\tilde{M}$ of $S$ in $V^\circ$ and a holomorphic function $\tilde{f}:\tilde{M}\to\d$ such that $\chi(V^\circ-\tilde{M})=0$ and $\|\tilde{f}-f\|<\epsilon$ on $M.$ Applying recursively this argument, we can find sequences  $\{V_n\}_{n \in \n}$  of compact regions in  $\mathcal{N}$ and holomorphic functions $\{f_n:V_n \to \d\}_{n \in \n},$ such that:
\begin{itemize}
 \item  $V_{n} \subset V_{n+1}^\circ,$ $\chi(V_{n+1}^\circ-V_n)=-1,$  $V_{n+1}^\circ-V_n$ has no bounded components in $V_{n+1}^\circ$ and $N:=\cup_{n \in \n} V_n$ is homeomorphic to $\mathcal{N},$ and
 \item $\|f_{n+1}-f_n\|<\epsilon 2^{-n-1}$ on $V_n$ for all $n,$ where $\epsilon=\max_{V_1}|f_1|-\min_{V_1}|f_1|>0.$
\end{itemize}
The sequence $\{f_n\}_{n \in \n}$ converges uniformly on compact subsets of $N$ to a non constant bounded holomorphic function $u:N \to \c.$ The proof is done.

We finish by proving a  Lorentzian version of Theorem \ref{th:harmonic} for  weakly complete maximal surfaces in the Lorentz-Minkowski 3-spacetime $\r_1^3$ with signature $(-,+,+).$ Recall that a conformal maximal immersion $X:M \to \r^3_1$ with singularities  is said to be {\em weakly complete} if the metric $\sum_{j=1}^3 |\phi_j|^2$ is complete on $M,$ where $\Phi=(\phi_1,\phi_2,\phi_3)$ are the Weierstrass data of $X$ (see \cite{u-y}).

\begin{corollary}\label{co:maxi}
Let $h:\Ncal\to\R$ be a non constant harmonic function.

Then there exist weakly complete conformal maximal immersions $Y=(Y_1,Y_2,Y_3):\Ncal\to\R_1^3$ and $Z=(Z_1,Z_2,Z_3):\Ncal\to\R_1^3$ with $Y_1=h=Z_2.$
\end{corollary}
\begin{proof}
Let $X=(X_1,X_2,X_3):\mathcal{N} \to \r^3$ be the immersion in Theorem \ref{th:harmonic} associated to $h$ and the group morphism $p:\mathcal{H}_1(\mathcal{N},\z)\to\r^3,$ $p(\gamma)=(0,0,\mbox{Im}\int_{\gamma} \partial h)$ for all $\gamma\in\mathcal{H}_1(\mathcal{N},\z).$ Labeling $X_j^*$ as the conjugate harmonic function of $X_j,$ $j=1,2,$ then $Y=(X_3,X_2^*,X_1^*):\mathcal{N}\to\r_1^3$ and $Z=(X_1^*,X_3,X_2):\mathcal{N}\to\r_1^3$ satisfy the conclusion of the corollary.
\end{proof}

%%%%%%%%%%%%%%%%%

%%%%%%%%%%%%%%%%%%%%%%%%%%%%%%%%%%%%%%%%%%%%%%%%%%%%%%%%%%%%%%%%%%%%%%%%%%%%%%%%%%%%%%%%%%%%%%%%%%%%%%%%%%%%%%%%%%%%%%%%%%%%%%%%%%%%%%%%%%%%

%%%%%%%%%%%%%%%%%%%%%%%%%%%%%%%%%%%%%%%%%%%%%%%%%%%%%%%%%%%%%%%%%%%%%%%%%%%%%%%%%%%%%%%%%%%%%%%%%%%%%%%%%%%%%%%%%%%%%%%%%%%%%%%%%%%%%%%%%%%%%%%%
\end{document}